\documentclass[11pt]{amsart}

\usepackage[notref,notcite]{showkeys}
\usepackage{hyperref}
\usepackage[all]{xy}
\usepackage{mathrsfs,bbm,stmaryrd,amscd,amssymb}

\newtheorem{thm}{Theorem}[section]
\newtheorem{prop}[thm]{Proposition}
\newtheorem{lem}[thm]{Lemma}
\newtheorem{cor}[thm]{Corollary}
\theoremstyle{definition}
\newtheorem{defin}[thm]{Definition}
\newtheorem{exam}[thm]{Example}

\theoremstyle{remark} 
\newtheorem{rmk}[thm]{Remark}

\numberwithin{equation}{section}

\def\La{\varLambda}      
\def\Ga{\varGamma}       \def\De{\varDelta}

\def\al{\alpha}       \def\de{\delta}
\def\be{\beta}

\def\vi{\varphi}      \def\io{\iota}
\def\si{\sigma}

\def\mA{\mathbb A} \def\mN{\mathbb N}

 \def\mQ{\mathbb Q}
 
 \def\mS{\mathbb S}

 \def\mZ{\mathbb Z}

\def\a1{\mathbbm1}

 \def\gN{\mathfrak n}
 
 \def\gP{\mathfrak p}
 \def\gQ{\mathfrak q}

\def\gM{\mathfrak m}

\def\dA{\mathfrak A}

 \def\dS{\mathfrak S}

\def\bA{\mathbf A} 
 
 \def\bP{\mathbf P}

\def\fK{\mathbf k}

\def\sF{\mathsf F} \def\sS{\mathsf S}

\def\cA{\mathscr A} 
\def\cB{\mathscr B} 
\def\cC{\mathscr C}

\def\qA{{\boldsymbol A}} 
 
\def\qC{{\boldsymbol C}} 
\def\qD{{\boldsymbol D}} \def\qQ{{\boldsymbol Q}}
 \def\qR{{\boldsymbol R}}
\def\qF{{\boldsymbol F}} \def\qS{{\boldsymbol S}}
 
\def\qH{{\boldsymbol H}}

\def\qK{{\boldsymbol K}}

\def\hqR{\hat{\qR}}			\def\hLa{\hat{\La}}
			\def\hqQ{\hat{\qQ}}	
\def\tLa{\tilde{\La}}			\def\hcA{\hat{\cA}}
				
\def\hA{\hat{A}}				\def\hB{\hat{B}}
\def\hDe{\hat{\De}}

\def\({\textup{(}}  \def\){\textup{)}}

\def\set#1{\left\{\,#1\,\right\}}
\def\setsuch#1#2{\left\{\,#1\mid #2\,\right\}}
\def\ssuch#1#2{\{\,#1\mid #2\,\}}

\def\mtr#1{\begin{pmatrix}#1\end{pmatrix}}

\def\lst#1#2{ #1_1 , #1_2 , \dots , #1_{#2} }

\def\mps{\mapsto}		\def\ito{\stackrel\sim\to}
\def\bap{\bigcap}        \def\bup{\bigcup}
         	
\def\spe{\supseteq}      \def\sbe{\subseteq}
\def\={\setminus}        \def\8{\infty}
\def\0{\varnothing}		\def\*{\otimes}  
\def\bop{\bigoplus}		\def\+{\oplus}
\def\xx{\times}				\def\wh{\widehat}
\def\div{\,|\,}

\def\ccd{{\boldsymbol\cdot}}
					
\def\Sb{\subsetplus}		
\def\xarr{\xrightarrow}

\def\ann{\mathop\mathrm{ann}\nolimits}
\def\co{\mathop\mathrm{co}\nolimits}
\def\End{\mathop\mathrm{End}\nolimits}
\def\ker{\mathop\mathrm{Ker}\nolimits}

\def\hos{\mathop\mathsf{Hos}\nolimits}
\def\hot{\mathop\mathsf{Hot}\nolimits}
\def\iso{\mathop\mathsf{iso}\nolimits}
\def\ind{\mathop\mathsf{ind}\nolimits}
\def\mdd{\mbox{-}\mathrm{mod}}
\def\proj{\mathrm{proj}\mbox{-}}

\def\lat{\mbox{-}\mathrm{lat}}

\def\cl{\mathop\mathrm{Cl}\nolimits}
\def\ob{\mathop\mathrm{ob}\nolimits}
\def\trs{\mathop\mathrm{tors}\nolimits}
\def\nil{\mathop\mathrm{nil}\nolimits}
\def\add{\mathop\mathrm{add}\nolimits}
\def\Mat{\mathop\mathrm{Mat}\nolimits}

\def\sw{\mathsf{SW}}

\def\iff{if and only if }
\def\lof{hom-finite}
\def\dvr{discrete valuation ring}
\def\ksc{Krull--Schmidt category}

\author{Yuriu~A.~Drozd}
\title{On $K_0$ of locally finite categories}
\address{Institute of Mathematics, NAS Ukraine,
Tereschenkivska 3, 01024 Kyiv, Ukraine}
\email{y.a.drozd@gmail.com,\,drozd@imath.kiev.ua}
\urladdr{https://www.imath.kiev.ua/$\sim$drozd}
\dedicatory{Dedicated to the memory of Andrei Roiter}

\begin{document}
\maketitle

\begin{abstract}
 We calculate the Grothendieck group $K_0(\cA)$, where $\cA$ is an additive category, locally finite over a Dedekind ring and
 satisfying some additional conditions. The main examples are categories of modules over finite algebras and the stable homotopy
 category $\mathsf{SW}$ of finite CW-complexes. We show that this group is a direct sum of a free group arising from localizations of the category
 $\cA$ and a group analogous to the groups of ideal classes of maximal orders. As a corollary, we obtain a new simple proof of the
 Freyd's theorem describing the group $K_0(\mathsf{SW})$.
\end{abstract}

\tableofcontents

\section*{Introduction}

 In this paper we study Grothendieck groups $K_0(\cA)$ of additive categories $\cA$ which are \emph{locally finite} over a Dedekind
 ring $\qR$. Among them there are categories of lattices over $\qR$-orders as well as the stable homotopy category $\sw$ of
 polyhedra (finite CW-complexes). Our main tool is the relation of such categories with the categories of projective modules (Lemma~\ref{11}), 
 which allows to study them ``piecewise,'' since usually only finitely many objects are involved in the considerations. 
 Perhaps, for the first time this idea was explored in the paper \cite{dk}.
 It replaces the usual technique using abelian or triangulated structure and makes the framework more flexible.
 So our investigation is quite parallel to the theory of integral representations and we have possibility to avail of the well developed 
 technique and results of this theory. Namely, we localize our categories and define \emph{genera},
 like in \cite[\S\,31]{cr1}. We establish ``local-global correspondence'' (Theorem~\ref{lg}) and prove analogues of the known results on
 genera, such as Jacobinski cancellation (Theorem~\ref{g4}) and Roiter addition theorem (Theorem~\ref{g2}\,(2)).
 It gives a basis for the calculation of $K_0(\cA)$ in the next section. Under some, not very restrictive, S-condition 
 we show that in the local case (when $\qR$ is a \dvr) this group is free and almost the same as the group of the
 adically completed category $\hcA$ (the difference is on the level of their rational envelopes), see Theorem~\ref{at2}. Finally, in
 Section~\ref{kg}, under a bit more restrictive Max-condition, we show that in the global case the group $\cA$ splits into a free 
 part $K_0(G\cA)$, which is of a purely local nature, and an analogue of the group of ideal classes $\bop_S\cl(S)$, where $S$ runs through 
 special objects called $S$-objects. They are analogues of maximal orders in the theory of integral representations and of 
 spheres in the stable homotopy theory. 
 
 As an application, we calculate the group $K_0(\La)$, where $\La$ is a hereditary order (Example~\ref{hered}), 
 and give a new simple proof of the Freyd's description of $K_0(\sw)$ (Example~\ref{sw}). Actually, the results of Section~\ref{kg}
 can be considered as a far-reaching generalization of the Freyd's theorem, which was the original incentive of our investigation.

\section{Generalities}
\label{general} 
 
 All categories and functors that we consider are supposed \emph{preadditive} and \emph{small}. An additive category $\cA$ is said to be 
 \emph{fully additive} if every idempotent morphism $e:A\to A$ in it splits, i.e. there are morphisms $A\xarr\pi B\xarr\io A$  
 such that $\io\pi=e$ and $\pi\io=1_B$. Then $A\simeq B\+C$, where $C$ is obtained in the same way from the idempotent 
 $1-e$. If $\cA$ is fully additive and $\dS\sbe\ob\cA$, we denote by $\add(\dS)$ the smallest full subcategory of $\cA$ 
 containing $\dS$ and closed under (finite) direct sums and direct summands. If $\dS=\{A\}$ consists of one object, we write 
 $\add(A)$ instead of $\add(\dS)$. Obviously, if $\dS=\set{\lst An}$ is finite, $\add(\dS)=\add\left(\bop_{i=1}^nA_i\right)$. 
 We write $A\Sb B$ if $A$ is a direct summand of $B$. It is known that every preadditive category $\cA$ can
 be embedded as a full subcategory into a fully additive category $\tilde{\cA}$ such that $\add(\ob\cA)=\tilde{\cA}$. This
 category is defined up to a natural equivalence, so we denote it by $\add\cA$ (see \cite[pp.60-61]{fr0}). We denote by
 $\End_\cA A$ the endomorphism ring $\cA(A,A)$ (though we write usual $\End_\La M$ instead of $\End_{\La\mdd}M$).
 
  To transform the study of categories to that of rings and modules, we use the following result which is actually a variant 
  of the Yoneda's lemma \cite{mac}. 
 
 \begin{lem}\label{11} 
  Let $\cA$ be a fully additive category, $C$ be an object of $\cA$ and $\La=\End_\cA C$. The map $A\mps\cA(C,A)$ induces an
  equivalence $\add C\ito \proj\La$, the category of finitely generated projective right $\La$-modules.
 \end{lem}
 \begin{proof}
  Note that every functor $F:\add C\to \cC$, where $\cC$ is a fully additive category, is completely determined (up to isomorphism) by
  its values on $C$ and on endomorphisms of $C$. As the functor $\cA(C,\_)$ maps $C$ to $\La$ and induces an isomorphism 
  $\End_\cA C\ito\End_\La\La\simeq\La$, it only remains to apply the Yoneda's lemma.
 \end{proof}
 
  Recall the definition of the Grothendieck group $K_0(\cA)$. 
 
 \begin{defin}\label{k0} 
   Let $\cA$ be a fully additive category. The \emph{Grothendieck group} $K_0(\cA)$ is a quotient of the free abelian group with the
   basis $\ob\cA$ by the subgroup generated by all elements of the form $A-B-C$, where $A\simeq B\+C$. We denote by $[A]$
   the image of $A$ in $K_0(\cA)$.
    
  One easily sees that $[A]=[B]$ \iff there is an object $C$ such that $A\+C\simeq B\+C$.\,%
\footnote{\,It is just the equivalence denoted by $A\equiv B$ in \cite{fr} or \cite{coh}. Note that we use the notation $\equiv$ for
another equivalence, see Corollary~\ref{g3}.}  
 \end{defin}
  
  We denote by $\iso\cA$ the set of isomorphism classes of objects from $\cA$ and by $\ind\cA$ its subset consisting of the classes
  of \emph{indecomposable} objects $A$, i.e. such that there are no decompositions $A\simeq B\+C$ with $B\ne0$ and $C\ne0$. We
  say that $\cA$ is a \emph{category with decomposition} if every object in it is isomorphic to a direct sum of indecomposable objects
  and a \emph{\ksc} if, moreover, such a decomposition is unique up to isomorphism and permutation of summands.
   
 A morphism $a\in\cA(A,B)$ is said to be \emph{essentially nilpotent} if for every sequence $b_1,b_2,b_3,\dots$ of elements of 
 $\cA(B,A)$ there is an integer $n$ such that $ab_1ab_2\dots ab_na=0$. The set of all essentially nilpotent morphisms $A\to B$ is 
 denoted by $\nil(A,B)$. One easily sees that $\nil\cA=\bup_{A,B}\nil(A,B)$ is an ideal in $\cA$ called the \emph{nilradical} of $\cA$. 
 If $\cA$ has one object, hence is identified with a ring $\qA$, $\nil\cA$ is the \emph{lower nil radical} 
 (or the \emph{prime radical}) of $\qA$ \cite{lam}. 
 If $\nil\cA=0$, the category $\cA$ is called \emph{semiprime}. It is known \cite[10.30]{lam} that if a ring $\La$ is left or
 right noetherian, $\nil\La$ is the maximal nilpotent ideal of $\La$ and contains all left and right nil-ideals.
 We denote by $\cA^0$ the quotient $\cA/\nil\cA$. This category has the same objects, but $\cA^0(A,B)=\cA(A,B)/\nil(A,B)$. 
 Obviously, it is semiprime. We denote by $A^0$ the object $A$ considered as an object of $\cA^0$ and
 by $\al^0$ the class of a morphism $\al$ in $\cA^0$. A morphism $\al$ is an isomorphism \iff so is $\al^0$, and any 
 idempotent from $\cA^0$ can be lifted to an idempotent in $\cA$. So the following results are evident. 
 
 \begin{prop}\label{12} 
   \begin{enumerate}
   \item  $\cA^0$ is fully additive \iff so is $\cA$.
   \item  $A\Sb B$ \iff $A^0\Sb B^0$.
   \item  $\iso\cA^0=\iso\cA$ and $\ind\cA^0=\ind\cA$.
   \item  $K_0(\cA^0)=K_0(\cA)$.
   \end{enumerate}
 \end{prop}
 
 Let $\qR$ be a commutative ring. Recall that an $\qR$\emph{-category} is a category $\cA$ such that all sets $\cA(A,B)$ are
 $\qR$-modules and the multiplication of morphisms is bilinear. A functor $F:\cA\to\cB$ between $\qR$-categories is called
 an $\qR$-\emph{functor} if all induced maps $\cA(A,B)\to\cB(FA,FB)$ are $\qR$-linear. 
  
 \begin{defin}\label{13} 
   An $\qR$-category $\cA$ is said to be \emph{\lof} if all $\qR$-modules $\cA(A,B)$ are finitely generated and \emph{finite} if, 
   moreover, there is a finite set of objects $\dS$ such that $\add\dS=\cA$. In particular, if $\La$ is an $\qR$-algebra, the category 
   $\proj\La$ is finite \iff $\La$ is a finitely generated $\qR$-module. Then we say that $\La$ is a \emph{finite $\qR$-algebra}.
 \end{defin}
 
  If the ring $\qR$ is noetherian, a \lof\ $\qR$-category is always a category with decomposition, but not necessarily a \ksc. 
  It is a \ksc\ if $\qR$ is a complete local ring \cite[23.3]{lam} (though this condition is not necessary). 	
  
  If $\qS$ is a commutative $\qR$-algebra and $\cA$ is an $\qR$-category, we define the $\qS$-category $\qS\*_\qR\cA$
  as the category with the same set of objects and the sets of morphisms ${(\qS\*_\qR\cA)(A,B)}=\qS\*_\qR\cA(A,B)$, 
  with the obvious multiplication.

 If $\qR$ is a domain, we denote by $\trs M$ the \emph{torsion submodule} of an $\qR$-module $M$, i.e. the set of all 
 periodic elements. If $\cA$ is an $\qR$-category and $A,B$ are its objects, we set $\trs(A,B)=\trs\cA(A,B)$ and 
 $\trs\cA=\bup_{A,B}\trs(A,B)$. It is an ideal in $\cA$ and the quotient category $\cA/\trs\cA$ is \emph{torsion free}, 
 i.e. all sets of morphisms in it are torsion free. We call an object $A$ \emph{torsion} if $\End_\cA A$ is torsion and
 \emph{torsion reduced} if $\trs(A,A)\sbe\nil(A,A)$. We denote by $\cA^t$ ($\cA^f$) the full subcategory of $\cA$
 consisting of torsion (respectively, torsion reduced) objects. If $\cA=\cA^t$ ($\cA=\cA^f$), we say that $\cA$ is
 \emph{torsion} (respectively, \emph{torsion reduced}). We call the $\qR$-category \emph{torsion free} if all modules
 $\cA(A,B)$ are torsion free. If $\cA$ is additive, it is enough to check endomorphism algebras $\End_\cA A$.

 \begin{lem}\label{14}
   Let $\qR$ be a Dedekind domain, $\La$ be a finite $\qR$-algebra. There are orthogonal idempotents $e_0$ and $e_1=1-e_0$ 
   in $\La$ such that, if we denote $\La_{ij}=e_i\La e_j$, then
 \begin{enumerate}
 \item  $\La_{ij}$ is torsion \(hence of finite length\) if $(i,j)\ne(1,1)$.
 
 \item  $\La_{01}\cup\La_{02}\cup\trs\La_{11}\sbe\nil\La$.
 
 \item  $\tLa=\La_{11}/\nil\La_{11}$ is semiprime and torsion free.
 \end{enumerate}
 We denote $e_0=e^t$ and $e_1=e^f$.
  \end{lem} 
  \begin{proof}
  Let $N=\nil\La$ and $\La^0=\La/N$. It contains no nilpotent ideals. Hence, every minimal left or right ideal of $\La^0$ is generated by 
  an idempotent.  As  $\trs\La^0$ is an ideal, it is semisimple and generated by an idempotent $\bar e_0$ both as left and as right ideal. 
  Let $\bar e_1=1-\bar e_0$. Then $\bar e_i\La^0\bar e_j=0$ if $i\ne j$, since it is torsion and $\trs\La^0=e_0\La^0e_0$. 
  Therefore, $\La^0=\La^0_0\xx\La^0_1$, where $\La^0_0$ is a semisimple ring and $\La^0_1$ is semiprime and torsion free. We take for 
  $e_i\in\La\ (i=0,1)$ a representatives of $\bar e_i$. Then $\La_{ij}$ is torsion for $(i,j)\ne(1,1)$, $\trs\La_{11}\sbe N$ and 
  $\tLa\simeq\La^0_1$, which proves (1)--(3). 
 \end{proof}
 
 Applying this lemma to the endomorphism rings of objects of a \lof\ $\qR$-category, we obtain the following results. 
 
 \begin{cor}\label{15} 
  Let $\qR$ be a Dedekind domain, $\cA$ be a fully additive \lof\ $\qR$-category. 
  \begin{enumerate}
  \item   If $A$ is torsion reduced and $B$ is torsion, then $\cA(A,B)\cup\cA(B,A)\sbe\nil\cA$.
  \item	 Every object $A\in\cA$ is a direct sum $A^t\+A^f$, where $A^t$ is torsion and $A^f$ is torsion reduced. 
  \item    If also $B\simeq B^t\+B^f$, where $B^t$ is torsion and $B^f$ is torsion reduced, then $A\simeq B$ \iff 
  $A^t\simeq B^t$ and $A^f\simeq B^f$.
  \item  Any indecomposable object is either torsion or torsion reduced.
  \item  $K_0(\cA)=K_0(\cA^t)\+K_0(\cA^f)$.
  \end{enumerate}
 \end{cor}
 \begin{proof}
  (1)  If $a:A\to B$, the left ideal $\cA(B,A)a$ of the ring $\End_\cA A$ is torsion, hence nilpotent, whence $a\in\nil(A,B)$. The proof
  for $\cA(B,A)$ is analogous. 
  
  (2) Let $\La=\End_\cA A$, $e^t$ and $e^f$ are as in Lemma~\ref{14}. They define a decomposition $A=A^t\+A^f$, where
  $\End_\cA A^t\simeq e^t\La e^t$ is torsion and $\End_\cA A^f\simeq e^f\La e^f$ is torsion reduced.
  
  (3) follows from (1), (4) follows from (2), and (5) follows from (2) and (3).
 \end{proof}
 
  Note that if $A$ is torsion, the ring $\End_\cA A$ is artinian. If $A$ is indecomposable, $\End_\cA A$ have no non-trivial idempotents,
  hence is local. Therefore, $\cA^t$ is a \ksc\ \cite[I.3.6]{bass} and $K_0(\cA^t)$ is a free group with a basis
  $\setsuch{[A]}{A\in\ind\cA^t}$. That is why, when studying Grothendieck groups, we can restrict to the case of torsion reduced
  categories.

  \section{Localization and genera}
  \label{genera}
 
  From now on $\qR$ denotes a Dedekind domain, $\qQ$ its field of fractions, $\max\qR$ the set of its maximal ideals,
  $\hqR_\gP$ the completion of the local ring $\qR_\gP$ in the $\gP$-adique topology and $\hqQ_\gP$ the field of fractions
  of $\hqR_\gP$. $\cA$ denotes a fully additive \lof\ $\qR$-category, and we write
  \begin{itemize}
  \item  $Q\cA$ for $\add(\qQ\*_\qR\cA)$,
  \item  $\cA_\gP$ for $\add(\qR_\gP\*_\qR\cA)$,
  \item  $\hcA_\gP$ for $\add(\hqR_\gP\*_\qR\cA)$,
  \item  $Q\hcA_\gP$ for $\add(\hqQ_\gP\*_\qR\cA)\simeq\add(\qQ\*_{\qR}\hcA_\gP)$.
  \end{itemize}
  Usually we denote by $QA,A_\gP,\hA_\gP,Q\hA_\gP$ the object $A$ considered as an object of the corresponding categories.
  Note that the operation $\add$ here is indeed necessary. It often happens, for instance, that $\ob\cA_\gP\ne\ssuch{A_\gP}{A\in\ob\cA}$.
  Following \cite{fr0}, we identify the objects of $\cA_\gP$ with the pairs $(A_\gP,e)$, where $A\in\ob\cA$ and $e$ is an idempotent from
  $\End_{\cA_\gP}A_\gP$. Then the set of morphisms $\cA_\gP((A_\gP,e),(B_\gP,f))$ is identified with $e(\cA_\gP(A_\gP,B_\gP))f$. 
  The same is valid for the objects and morphisms of $Q\cA,\,\hcA_\gP$ and $Q\hcA_\gP$.\!%
  \footnote{\,In the important case, when the Max-condition \ref{max-cond} is satisfied, the objects from $Q\cA$ are exactly $QA$ 
  with $A\in\ob\cA$, though analogous equality can still be wrong for $\cA_\gP$ and, all the more, for $\hcA_\gP$.}  

  \begin{defin}\label{g1} 
    Let $A$ be an object from $\cA$, 
   \[    
    G(A)=\setsuch{B\in\ob\cA}{B_\gP\simeq A_\gP \text{ for all } \gP\in\max\qR}.
   \] 
    We call $G(A)$ the \emph{genus} of $A$. If $G(B)=G(A)$ (or, the same, $B\in G(A)$), we say that $A$ and $B$ \emph{are
    of the same genus}. The cardinality of $G(A)$ is denoted by $g(A)$.
  \end{defin}
  
  We denote by $\qR\lat$ the category of $\qR$\emph{-lattices}, i.e. finitely generated torsion free $\qR$-modules.
  Such a lattice $M$ is always considered as a submodule of the finite dimensional $\qQ$-vector space $QM=\qQ\*_\qR M$.
  If $\La$ is an $\qR$-algebra we denote by $\La\lat$ the category of $\La$-modules which are $\qR$-lattices as $\qR$-modules
  and call them $\La$\emph{-lattices}. $\La\lat$ is a \lof\ fully additive torsion free $\qR$-category. If $\La$ is itself an $\qR$-lattice, 
  it is called an $\qR$\emph{-order}. Then it is a subalgebra in the finite dimensional $\qQ$-algebra $Q\La=\qQ\*_\qR\La$ 
   and they say that $\La$ is \emph{an $\qR$-order in} $Q\La$.
   An \emph{overring} of  $\La$ is an $\qR$-order $\La'$ such that $\La\sbe\La'\sbe Q\La$. If $\La$ has no proper overrings, 
   it is called a \emph{maximal order}. An overring that is a maximal order is called a \emph{maximal overring}. 
   If the $\qQ$-algebra $Q\La$ is separable or if $Q\La$ is semisimple and $\qR$ is an \emph{excellent} ring \cite{ma}, 
   $\La$ has maximal overrings and all of them are Morita equivalent \cite{cr1}.
  
In what follows we use the results on the ``local-global correspondence'' from the theory of orders and lattices \cite{cr1}. Actually we
need more refined versions, so we formulate them here.

 \begin{lem}\label{lg1} 
   Let $M,N$ are $\La$-lattices such that $N\spe M\spe aN$ for some non-zero $a\in\qR$, $\lst \gP r$ are all prime ideals of $\qR$ 
   containing $a$. There are $\La$-lattices $\lst Mr$ such that 
  \begin{itemize}
     \item  $(M_i)_{\gP_i}= M_{\gP_i}$ and $(M_i)_\gQ= N_\gQ$ if $\gQ\ne\gP_i$;
     \item  $\bop_{i=1}^rM_i\simeq M\+(r-1)N$.
  \end{itemize}   
  In particular, if $M$ and $N$ are projective, so are all $M_i$.
 \end{lem}
 \begin{proof}
  $N/M$ is an $\qR$-module of finite length and $\lst\gP r$  are all prime ideals associated with $N/M$. Hence 
  there are submodules $\lst Mr$ of $N$ such that
  \begin{itemize}
  \item $M_i$ is $\gP_i$-primary, i.e. $M_i\spe\gP_i^{k_i}N$;
  \item $\bap_{i=1}^rM_i=M$;
  \item $M_i\not\sbe M_i'=\bap_{j\ne i}M_j$.
  \end{itemize}
  (see \cite[Sec.\,8]{ma}). As $(\prod_{j\ne i}\gP_i^{k_i})N\sbe M'_i$ and $\gP^{k_i}_i+\prod_{j\ne i}\gP_j^{k_j}=\qR$, it implies 
  that $N=M_i+M'_i$. As $M_i$ is $\gP_i$-primary, $(M_i)_\gQ=N_{\gQ}$ for $\gQ\ne\gP_i$. 
  Moreover, $M_i/M\simeq N/M'_i$ is annihilated by $\prod_{j\ne i}\gP_j^{k_j}$, hence $(M_i)_{\gP_i}=M_{\gP_i}$.
  Consider the map $\vi:\bop_{i=1}^rM_i\to(r-1)N$ such that $\vi(\lst ur)=(u_1+u_2,u_2+u_3,\dots,u_{r-1}+u_r)$. For any $v\in N$
  there are $u\in M_1$ and $u'\in M'_1$ such that $u+u'=v$, whence $(v,0,0,\dots,0)=\vi(u,u',-u',u',\dots,(-1)^ru')$. Just in the same 
  way all components of $(r-1)N$ are in the image, so $\vi$ is surjective. Moreover, $\vi(\lst ur)=0$ means that $u_i=-u_{i-1}$ for 
  $1<i\le r$, so this row is of the form $(u,-u,u,\dots,(-1)^{r-1}u)$, where $u\in\bap_{i=1}^rM_i=M$. Thus we obtain an exact 
  sequence $0\to M\to \bop_{i=1}^nM_i\to(r-1)N\to0$. One easily sees that its localization
  $0\to M_\gP\to \bop_{i=1}^n(M_i)_\gP\to(r-1)N_\gP\to0$ splits for every prime $\gP$. Therefore, this sequence splits 
  \cite[3.20]{rei} and $\bop_{i=1}^rM_i\simeq M\+(r-1)N$.
  \end{proof}
  
  \begin{lem}\label{lg2} 
   Let $\lst\gP r$ be different prime ideals of $\qR$, $M_i\ (1\le i\le r)$ be a $\La_{\gP_i}$-lattice, $N$ be a $\La$-lattice and
   $QM_i= QN$ for all $i$. There is a $\La$-lattice $M$ such that $M_{\gP_i}= M_i$ for all $i$ and 
   $M_\gQ= N_{\gQ}$ if $\gQ\notin\set{\lst\gP r}$. If the modules $\lst Mr,N$ are projective, so is $M$.
  \end{lem}
  \begin{proof}
  First suppose that $M_i\sbe N_{\gP_i}$ for all $i$.
   Set $M'_i=N\cap M_i$. Then $(M'_i)_{\gP_i}=M_i$ and $(M'_i)_\gQ=N_\gQ$ if $\gQ\ne\gP_i$, 
   so $M'_i$ is a $\gP_i$-primary submodule in $N$. Set $M=\bap_{i=1}^rM'_i$.
   The same arguments as in the preceding proof show that $M_{\gP_i}=M_i$ and $M_\gQ=N_\gQ$ if
   $\gQ\notin\set{\lst\gP r}$. 
   
   In general situation find a non-zero $a\in\qR$ such that $aM_i\sbe N_{\gP_i}$ for all $i$.
   Let $\lst\gQ s$ be all prime ideals, different from $\lst\gP r$, that contain $a$. As we have just proved, there is a lattice 
   $M^a$ such that $M^a_{\gP_i}=aM_i\ (1\le i\le r)$, $M^a_{\gQ_j}=aN_{\gQ_j}\ (1\le j\le s)$ and $M^a_\gQ=N_{\gQ}$
   if $\gQ\notin\set{\lst\gP r,\lst\gQ s}$. Then we can set $M=a^{-1}M^a$.
   If $\lst Mr,N$ are projective, then $M_\gP$ is $\La_\gP$-projective for all prime $\gP$, 
   thus $M$ is projective \cite[3.23]{rei}.
  \end{proof}
  
  \begin{lem}\label{lg3} 
   Let $M,N$ be $\La$-lattices, $\lst\gP r$ be different prime ideals of $\qR$ and homomorphisms $\al:M\to N$ and
   $\be_i:M_{\gP_i}\to N_{\gP_i}$ be given such that $Q\be_i=Q\al$ for all $i$. There is a unique homomorphism
   $\be:M\to N$ such that $\be_{\gP_i}=\be_i$ and $\be_\gQ=\al_\gQ$ if $\gQ\notin\set{\lst \gP r}$.
  \end{lem}
  \begin{proof} 
   Let $A=\setsuch{(u,\al(u))}{u\in M}$ be the graph of $\al$, $B_i$ be the graph of $\be_i$. They are submodules in
   $QM\+QN$ and $QA= QB_i$. By Lemma~\ref{lg2} there is a lattice $B$ such that $B_{\gP_i}=B_i$ 
   and $B_\gQ=A_\gQ$ if $\gQ\notin\set{\lst\gP r}$. As all projections $A\to QM$ and $B_i\to QM$ are monomorphisms,
   so is the projection $B\to QM$. Therefore, $B$ is the graph of a homomorphism $\be$ such that
   $\be_{\gP_i}=\be_i$ and $\be_\gQ=\al_\gQ$ if $\gQ\notin\set{\lst \gP r}$. As $M=\bap_{\gP\in\max\qR}M_\gP$,
   $\be$ is unique.
  \end{proof}
  
  In what follows $\cA$ is a \lof\ fully additive $\qR$-category.
 One easily checks that
 \begin{align*}
   & \nil(QA,QB)=Q\nil(A,B),\\
   & \nil(A_\gP,B_\gP)=\nil(A,B)_\gP,\\
   \intertext{whence}
   & (Q\cA)^0=Q\cA^0\, \text{ and }\, (\cA_\gP)^0=(\cA^0)_\gP.\\
   \intertext{If the ring $\qR$ is \emph{excellent} \cite{ma}, then also}
   & \nil(\hA_\gP,\hB_\gP)=\wh{\nil(A,B)}_\gP,  \\
   & \nil(Q\hA_\gP,Q\hB_\gP)=Q\wh{\nil(A,B)}_\gP, \\
   \intertext{whence}
   & (\hcA_\gP)^0=(\wh{\cA^0})_\gP\, \text{ and }\, (Q\hcA_\gP)^0=Q(\wh{\cA^0})_\gP.
 \end{align*}  
 These equalities for completions are also valid if $\cA$ satisfies the Max-condition, see Proposition~\ref{max-s}\,(1).
 
 Following \cite{ma}, we call an object $A$ \emph{$\gP$-coprimary} if it is torsion and $A_\gQ=0$ for $\gQ\ne\gP$.
 Note that, if $\qA$ is a finite $\qR$-algebra which is torsion as $\qR$-module, $\qA\simeq\prod_{i=1}^r\qA_{\gP_i}$ for some
 prime ideals $\lst\gP r$. If $\qA=\End_\cA A$, where $A$ is a torsion object, it gives a decomposition $A\simeq\bop_{i=1}^rA_i$,
 where $A_i$ is $\gP_i$-coprimary. Such decomposition is unique, since $\cA(A,B)=0$ if $A$ is
 $\gP$-coprimary and $B$ is $\gQ$-coprimary for $\gP\ne\gQ$. It implies that 
 \[
  K_0(\cA^t)\simeq{\bop}_{\gP\in\max\qR} K_0(\cA^\gP),
 \]
  where $\cA^\gP$ is the subcategory of $\gP$-coprimary objects. So from now on we can suppose that the category $\cA$ is torsion
  reduced. Then $\End_{\cA^0}A^0$ is a semiprimary $\qR$-order for every object $A$ and $\cA^0(A^0,B^0)$ is a (right)
  $End_{\cA^0}A^0$-lattice. 
  
  The next theorem provides a background for the theory of genera.
  
  \begin{thm}\label{lg} 
  Let $\cA$ be a fully additive \lof\ $\qR$-category, $\gN=\nil\cA$.
  \begin{enumerate}
  \item    Let $A,B$ be torsion reduced objects and $A\xarr\al B\xarr\be A$ be such morphisms that $\be\al\equiv a1_A\,\pmod\gN$ 
  and $\al\be\equiv a1_B\,\pmod\gN$ for some non-zero $a\in\qR$. 
  If $\lst\gP r$ are all prime ideals of $\qR$ such that $a\in\gP_i$, there are torsion reduced objects $\lst Ar$ such that 
  \begin{enumerate}
     \item  $(A_i)_{\gP_i}\simeq A_{\gP_i}$ and $(A_i)_\gQ\simeq B_\gQ$ if $\gQ\ne\gP_i$;
     \item  $\bop_{i=1}^rA_i\simeq A\+(r-1)B$.
  \end{enumerate}
  \smallskip
  \emph{Note that such an element $a$ exists \iff $QA\simeq QB$.}
  
  \medskip
  \item     Let $\lst\gP r$ be different prime ideals of $\qR$, $A_i\ (1\le i\le r)$ be a torsion reduced object of $\cA_{\gP_i}$, $B$ be a torsion 
  reduced object of $\cA$ and $QA_i\simeq QB$ for all $i$. There is a torsion reduced object $A$ such that 
  $A_{\gP_i}\simeq A_i\ (1\le i\le r)$ and $A_\gQ\simeq B_\gQ$ if $\gQ\notin\set{\lst\gP r}$.
   
   \medskip
   \item  Suppose that $\cA$ is torsion free. Let $\lst\gP r$ be different prime ideals of $\qR$, $\al\in\cA(A,B)$ and
   $\be_i\in\cA_{\gP_i}(A_{\gP_i},B_{\gP_i})$ be such that $Q\be_i=Q\al$ for all $i$. There is a morphism $\be:A\to B$
   such that $\be_{\gP_i}=\be_i$ and $\be_\gQ=\al_\gQ$ if $\gQ\notin\set{\lst \gP r}$.
  \end{enumerate}
  \end{thm}
  \noindent
  Note that the genera of the objects $A_i$ in (1) and of the object $A$ in (2) are uniquely defined.
  \begin{proof}
  (1) Replacing $\cA$ by $(\cA^f)^0$ we can suppose that $\cA$ is torsion reduced and semiprime, hence torsion free. 
  Let $C=A\+B$, $\La=\End_\cA C$, $M=\cA(C,A)$,\,$N=\cA(C,B)$. $M$ and $N$ are projective $\La$-lattices and the 
  multiplications by $\al$ and $\be$ give homomorphisms $M\xarr{\al\ccd}N\xarr{\be\ccd}M$ such that 
  $(\al\ccd)(\be\ccd)=a1_N$,\,$(\be\ccd)(\al\ccd)=a1_M$. 
  Hence $\al\ccd$ and $\be\ccd$ are monomorphisms, so we can suppose that $\al\ccd$ is an embedding $M\sbe N$ such that 
  $M\spe aN$. Now we are in the situation of Lemma~\ref{lg1}. Therefore, there are $\La$-lattices $M_i$ such that
    \begin{itemize}
     \item  $(M_i)_{\gP_i}= M_{\gP_i}$ and $(M_i)_\gQ= N_\gQ$ if $\gQ\ne\gP_i$;
     \item  $\bop_{i=1}^rM_i\simeq M\+(r-1)N$.
  \end{itemize}   
  As $M$ and $N$ are projective, so are $M_i$. By Lemma~\ref{11}, there are objects $A_i$ such that $M_i=\cA(C,A_i)$.
  Then $A_i$ satisfy conditions (a) and (b).
  
  (2) follows in the same way from Lemma~\ref{lg2} if we choose an object $\tilde{A}$ such that $A_i\Sb\tilde{A}_{\gP_i}$ for all $i$, 
  set $C=\tilde{A}\+B$ and apply the functor $\cA(C,\_)$.
  
  (3) is deduced in the same way from Lemma~\ref{lg3}, setting $C=A\+B$.
  \end{proof}
  
   We transfer to a \lof\ $\qR$-category $\cA$ some results on genera of lattices over orders. 
  
  \begin{thm}\label{g2} 
   \begin{enumerate}      
   \item   If $A_\gP\Sb B_\gP$ for all prime $\gP$, there is an object $A'\in G(A)$ such that $A'\Sb B$. 
   
   \item\emph{(Roiter addition theorem)}  Let $QA\in\add QB$, $A'\in G(A)$. There is an object $B'\in G(B)$ such that $A\+B\simeq A'\+B'$.
   
   \item  If $A'\in G(A)$, then $A'\Sb A\+A$.
   
   \item  Let $A',A''\in G(A)$. There is $A'''\in G(A)$ such that $A'\+A''\simeq A\+A'''$.
   \end{enumerate}
  \end{thm}  
  \begin{proof}  In view of Proposition~\ref{12}, we may suppose that $\cA$ is semiprime.
  If $A$ is torsion, all claims are trivial. So we may suppose that $\cA$ is torsion free.
 We use Lemma~\ref{11} for $C=A\+B$. Set $\La=\End_\cA C$, $M=\cA(C,A)$ and $N=\cA(C,B)$. 
 Then $\La$ is a semiprime order and $M,N$ are projective $\La$-lattices. 
  
  (1) $M_\gP\Sb N_\gP$ for all $\gP$, hence there is a lattice $M'\in G(M)$ such that $M'\Sb N$ \cite[31.12]{cr1}. 
  This lattice is projective, as so is $M$ \cite[3.23]{rei}, hence there is an object $A'$ such that $\cA(C,A')\simeq M'$. 
  Then $A'\in G(A)$ and $A'\Sb B$.
  
  \smallskip
  (2) In this situation $QC\in\add QB$, hence $Q\La\in\add QN$, so $N$ is a faithful $\La$-module. Set $M'=\cA(C,A')$, then
  $M'\in G(M)$. By Roiter addition theorem \cite[31.28]{rei}, there is a lattice $N'\in G(N)$ such that $M\+N\simeq M'\+N'$. Again, $N'$ is 
  projective, hence $N'\simeq\cA(C,B')$, where $B'\in G(B)$ and $A\+B\simeq A'\+B'$.
  
  \smallskip
  (3) and (4) follow from (2) (in (3) set $A=B$).
  \end{proof}
  
  \begin{cor}\label{g3} 
 Define an equivalence relation $\equiv$ on the set $G(A)$ such that $A'\equiv A''$ means $A'\+A\simeq A''\+A$. Denote by $\cl(A)$ the set
  of equivalence classes under this relation and by $c(A')$ the class of $A'$ in $\cl(A)$. Define an algebraic operation $+$ on $\cl(A)$
  setting $c(A')+c(A'')=c(A''')$ if $A'\+A''\simeq A\+A'''$. Then
   $(\cl(A),+)$ is an abelian group and $c(A)$ is its neutral element.
 
  \smallskip  
   \emph{Note that this group does not depend on the choice of an object $A$ in the genus (see Remark~\ref{g6} below). 
   We call it the \emph{class group} of the object $A$ or of the genus $G(A)$ and denote its cardinality by $cl(A)$. If $\cA=\La\lat$,
   it is the group of $\be$-classes of the genus $G(A)$ from \cite{roi} (see also \cite[35.5]{rei} in the case of maximal orders).}
  \end{cor}
  
  Now we prove the \emph{cancellation theorem} for genera.
  
  \begin{thm}[Jacobinski cancellation theorem]\label{g4} 
    Let $A$ be such an object that $Q\End_\cA A/\nil Q\End_\cA A\simeq \prod_{i=1}^m\Mat(n_i,\qD_i)$, where $\qD_i$ are skewfields, and
    for every $i$ either the skewfield $\qD_i$ is commutative or $n_i>1$ \(or both\). If $A',A''\in G(A)$, then $A'\equiv A''$ \iff $A'\simeq A''$.
  \end{thm}
  \begin{proof}
   As before, we may suppose that $\cA$ is torsion reduced and semiprimary. Then $\La=\End_\cA A$ is a semiprimary $\qR$-order, 
   $M'=\cA(A,A')$ and $M''=\cA(A,A'')$ are projective  $\La$-lattices, $Q\La\simeq Q\End_\cA A$, $M',M''\in G(\La)$ and, if $A'\equiv A''$, also 
   $M'\equiv M''$. By the Jacobinski cancellation theorem \cite[51.24]{cr2}, if $M'\equiv M''$, then $M'\simeq M''$, hence $A'\simeq A''$.
  \end{proof}
  
  \begin{cor}\label{g5} 
    Let $A,A'$ and $B$ be such objects that $QA\in\add QB$ and $A'\in G(A)$.
  \begin{enumerate}
   \item  If $A''\in G(A)$ and $A'\equiv A''$, then $A'\+B\simeq A''\+B$.
   \item If $A\+B\simeq A'\+B'$, where $B'\equiv B$ in $G(B)$, then $A\+B\simeq A'\+B$.
   \item If $A\+B\simeq A'\+B$, where $B\in\add A$, then $A\equiv A'$ in $G(A)$.
  \end{enumerate}    
  \end{cor}
  \begin{proof}
  (1)   If $A'\+A\simeq A''\+A$, then $(A'\+B)\+(A\+B)\simeq(A''\+B)\+(A\+B)$. As $QA\in\add QB$, the object $A\+B$ satisfies the conditions
   of Theorem~\ref{g4} (namely, all $n_i>1$), whence $A'\+B\simeq A''\+B$.
   
   (2)  Let $B\+B\simeq B'\+B$. Then $(A\+B)\+(A\+B)\simeq(A'\+B)\+(A\+B)$. Again the object $A\+B$ satisfies the conditions
   of Theorem~\ref{g4}, whence $A\+B\simeq A'\+B$.
   
   (3) immediately reduces to the case $B=nA$, where it is proved by an easy induction.
  \end{proof}
  
  \begin{rmk}\label{g6} 
    Item (1) of this corollary immediately implies that the relation $\equiv$ on the genus $G(A)$ does not depend on the choice of 
    the object $A$ in this genus: just take for $B$ any object from $G(A)$.
  \end{rmk}
     
   \subsection{Arithmetic case}
   \label{ar}
   
   We call a Dedekind ring $\qR$ \emph{arithmetic} if its field of fractions $\qQ$ is a global field, i.e. either a field of algebraic numbers 
   or a field of algebraic functions of one variable over a finite field. Then the preceding results can be precised. 
   
   \begin{thm}\label{ar1} 
    \begin{enumerate}
    \item  $g(A)<\8$ for every object $A$.
    \item  If $QA\in\add QB$, then $g(A\+B)\le cl(B)$.
    \end{enumerate}
   \end{thm}
   \begin{proof}
     (1) is obviously reduced to the semiprime and torsion free case, where it follows from the Jordan--Zassenhaus theorem \cite[26.4]{rei}.
     
     (2) By Theorem~\ref{g2}\,(1), for every $C\in G(A\+B)$ there is $A'\in G(A)$ such that $C\simeq A'\+B'$. Then $B'\in G(B)$. 
     By item (2) of the same theorem, there is $B''\in G(A)$ such that $C\simeq A\+B''$. Moreover, if $B'\equiv B''$ in $G(B)$, i.e.
     $B'\+B\simeq B''\+B$, then $(A\+B')\+(A\+B)\simeq(A\+B'')\+(A\+B)$, whence $A\+B'\simeq A\+B''$ by Theorem~\ref{g4}.
   \end{proof}
   
   If $v$ is a valuation of the field $\qQ$, we denote by $\hqQ_v$ the completion of $\qQ$ with respect to $v$. We say that $v$ is
   \emph{infinite with respect to $\qR$} if it is not equivalent to the $\gP$-adic valuation for any $\gP\in\max\qR$.
   Let $\qD$ be a finite dimensional division algebra over the field $\qQ$. We say that $\qD$ \emph{satisfies the Eichler condition} 
   if for some valuation $v$ of the field $\qQ$ that is infinite with respect to $\qR$ the $\hqQ_v$-algebra $\hqQ_v\*_\qQ\qD$ 
   is not a product of skewfields. Note that if $\qR$ is the ring of algebraic integers from $\qQ$, the only exceptions are when the center 
   $\qK$ of $\qD$ is a totally real field, $\dim_\qK\qD=4$ and $\hat{\qD}_w$ is the skewfield of quaternions for every infinite valuation $w$
   of the field $\qK$.
     
   \begin{thm}\label{ac2} 
     Jacobinski cancellation theorem~\ref{g4} remains valid if the condition $n_i>1$ holds true for those $\qD_i$ that do not satisfy the Eichler condition.
   \end{thm}
   \begin{proof}
    It is just the situation when the Jacobinski cancellation theorem is valid in the arithmetic case \cite[51.24]{cr2}. So the same proof can 
    be applied.
   \end{proof}
   
   \section{$K_0$, local case.}
   \label{loc} 
   
   In this section $\qR$ is a \dvr\ with the maximal ideal $\gM$, $\hqR$ is its $\gM$-adic completion and $\hqQ$ is the field of fractions
   of $\hqR$. We have a diagram of categories and functors, commutative up to isomorphism,
   \begin{equation}\label{eloc1} 
    \begin{CD}
        \cA @>^\wedge>> \hcA \\  @VQVV  @VVQV \\ Q\cA  @>^\wedge>> Q\hcA
    \end{CD}
   \end{equation}
  The $K_0$-groups form an analogous diagram
      \begin{equation}\label{eloc2} 
    \begin{CD}
        K_0(\cA) @>>> K_0(\hcA) \\  @VVV  @VVV \\ K_0(Q\cA)  @>>> K_0(Q\hcA)
    \end{CD}
   \end{equation}
   The categories $Q\cA,\hcA$ and $Q\hcA$ are Krull--Schmidt categories, so their $K_0$-groups are free and their bases consist of 
   classes of indecomposable objects. On the other hand, the category $\cA$ need not be a \ksc\ (see Example~\ref{nks} below).
   Nevertheless, as $A\simeq B$ \iff $\hat{A}\simeq\hat{B}$ and cancellation holds true in $\hcA$, it also holds true in $\cA$.
   
   The category $\cA$ can be reconstructed from the other components of the diagram \eqref{eloc1}.
   
   \begin{thm}\label{loc1}  
   The category $\cA$ is equivalent to the \emph{pull-back} \(or \emph{recollement}\) $Q\cA\xx_{Q\hcA}\hcA$ 
   \cite[VI.1]{gab} of the categories $Q\cA$ and $\hcA$ over $Q\hcA$.\!%
   \footnote{\,Actually, it means that the diagram \eqref{eloc1} is \emph{cartesian} as a diagram of categories and functors.}
   \end{thm}
  
   Recall that the category $Q\cA\xx_{Q\hcA}\hcA$ consists of triples $(V,\hat{A},\si)$, where $V\in Q\cA,\, \hat{A}\in\hcA$ and 
   $\si:\hat{V}\ito Q\hat{A}$, and a morphism $(V,\hat{A},\si)\to(V',\hat{A'},\si')$ is a pair $\be:V\to V',\,\al:\hat{A}\to \hat{A'}$ such that 
   $(Q\al)\si=\si'\hat{\be}$.
   
   \begin{proof}
    We define a functor $\sF:\cA\to Q\cA\xx_{Q\hcA}\hcA$ setting $\sF(A)=(QA,\hat{A},\si)$, where $\si:\widehat{QA}\ito Q\hat{A}$ 
    comes from the identity morphism of $A$, and $\sF(\al)=(Q\al,\hat{\al})$. Note that the diagram
    \[
     \begin{CD}
       M @>>> \hat{M} \\ @VVV @VVV \\ QM @>>> Q\hat{M}
     \end{CD}
    \]
    is cartesian for every finitely generated $\qR$-module $M$. It implies that $\sF$ is fully faithful. So it remains to show that it is dense.
    
	An object from $Q\cA$ is a pair $B_1=(A_1,e_1)$ where $A_1\in\ob\cA$ and $e_1$ is an idempotent in $\End_{Q\cA}A_1$.
    An object from $\hcA$ is a pair $B_2=(A_2,e_2)$ where $A_2\in\ob\cA$ and $e_2$ is an idempotent in $\End_{\hcA}A_2,$. 
    Setting $B=A_1\+A_2$, we can replace both $A_1$ and $A_2$ by $B$, so suppose that $B_1=(B,e_1)$ and $B_2=(B,e_2)$. 
    An isomorphism $\hat{B}_1\ito QB_2$ is then given by an automorphism $\si$ of $Q\hat{B}$ such that $\si e_1\si^{-1}=e_2$.
    Let $\La=\End_\cA B$. As $(\hat{\La})^\xx(Q\La)^\xx=(Q\hat{\La})^\xx$, there are automorphisms $\si_1$ of
    $QB$ and $\si_2$ of $\hB$ such that $\si=\si_2\si_1$, whence $\si_1e_1\si_1^{-1}=\si_2^{-1}e_2\si_2$. It implies that there is an idempotent
    $e\in\La$ such that its image in $Q\La$ is $\si_1e_1\si_1^{-1}$ and its image in $\hat{\La}$ is $\si_2^{-1}e_2\si_2$. This idempotent
    arises from a direct summand $A\Sb B$ such that $B_1\simeq QA$, $B_2\simeq \hA$ and $(B_1,B_2,\si)\simeq \sF(A)$, which
    accomplishes the proof.
    \end{proof}
    
    \begin{cor}\label{loc2} 
    \begin{enumerate}
    \item        The diagram \eqref{eloc2} is cartesian.
    \item  The group $K_0(\cA)$ is free.
    \end{enumerate}
    \end{cor}
    \begin{proof}
     (1) follows immediately from Theorem~\ref{loc1}. As the groups $K_0(Q\cA)$ and $K_0(\hcA)$ are free, it implies (2).
    \end{proof}
    
    There is one important case, when the generators of the group $K_0(\cA)$ can be explicitly calculated. Fortunately, most 
    examples that occur in applications are of this sort. (As a rule, they even satisfy much more restrictive \emph{Max-condition},
    see Definition~\ref{max-cond}.)
    
    \begin{defin}[S-condition]\label{s-cond} 
      We say that a \lof\ $\qR$-category $\cA$ \emph{satisfies the S-condition} if for every indecomposable object 
      $U\in\ind Q\hcA$ there is an object $S\in\ob\hcA$ such that $U\simeq QS$. Obviously, then $S$ is also indecomposable.
      
        If this condition is satisfied, we fix an object $\sS(U)$ such that $Q\sS(U)\simeq U$ for every $U\in\ind Q\hcA$
    and set $\sS(V)=\bop_{i=1}^m\sS(U_i)$ if $V\simeq\bop_{i=1}^mU_i$, where $U_i\in\ind Q\hcA$.
    
      We set $\mS(\hcA)=\ssuch{\sS(U)}{U\in\ind Q\hcA}$ and $\mA(\hcA)=\ind\hcA\=\mS(\hcA)$.
    \end{defin}     
     
    In the rest of this section we suppose that $\cA$ satisfies the S-condition and we use the notations of Definition~\ref{s-cond}.  
    Fix some more notations and terms. Note that $Q\cA^0$ is a semisimple category, i.e. $Q\cA^0(W,W')=0$ if $W,W'\in\ind Q\cA$
    and $W\not\simeq W'$. It implies that $Q\cA(W,W')\in\nil Q\cA$, whence also $Q\hcA(\hat{W},\hat{W'})\in\nil Q\hcA$. Therefore,
    $\hat{W}$ and $\hat{W'}$ have no isomorphic direct summands.
    
   \begin{defin}\label{at0} 
           \def\lll{\mbox{\Large$\lceil$}}	   \def\rrr{\mbox{\Large$\rceil$}}
    Let $U\in\ind Q\hcA,\,V\in\ob\hcA,\,W\in\ind Q\cA$.
     \begin{enumerate}
     \item  The multiplicities $\mu(U,V)$ are defined from the decomposition $V\simeq\bop_{U\in\ind Q\hcA}\mu(U,V)U$.
     
     \item  We set $\mu(U)=\mu(U,\hat{W})$ if $\mu(U,\hat{W})\ne0$. Note that there is exactly one $W$ with this property.
     
     \item  If $\mu(U)\div \mu(U,V)$ for all $U\in\ind Q\cA$, we define $S_V$
     as an object from $\cA$ such that $\widehat{S_V}\simeq\sS(V)$ (it exists by Theorem~\ref{loc1}). Note that if this condition
     is not satisfied there is no object $V_a$ in $Q\cA$ such that $V\simeq\widehat{V_a}$.
      
    \item  We set $\sS(W)=S_{\widehat{W}}$, denote by $\mS(\cA)$ the set $\{\,\sS(W)\div W\in\ind Q\cA\,\}$ and call these 
    objects \emph{S-objects}. 
     
   \item Let $V\in\ob Q\hcA,\,U\in\ind Q\hcA$. We denote by $\de(U,V)$ the smallest non-negative integer such that 
   $\mu(U)\div \mu(U,V)+\de(U,V)$.
   
   \item For an object $C\in\mA(\hcA)$ we set $\tilde{C}=C\+\left(\bop_{U\in\ind \hcA}\de(U,QC)\sS(U)\right)$.
   
   \item We call an object $A\in\ob\qA$ \emph{atomic} if $\hat{A}\simeq\tilde{C}$ for some $C\in\mA(\hcA)$.  Note that such
   object $A$ exists and is unique up to isomorphism for every $C\in\mA(\hcA)$ by Theorem~\ref{loc1}. In this case we call $C$ 
   \emph{the core} of $A$ and denote it by $\co(A)$. We denote by $\mA(\cA)$ the set of isomorphism classes of atomic objects. 
     \end{enumerate}
   \end{defin}
   
   The following properties immediately follow from definitions.
   
   \begin{prop}\label{at1} 
     \begin{enumerate}
     \item   Any atomic object or S-object is indecomposable.
     
     \item  Two atomic objects are isomorphic \iff their cores are isomorphic.
     
     \item  S-objects $S$ and $S'$ are isomorphic \iff $QS\simeq QS'$.
     
     \item  Neither atomic object is isomorphic to an S-object.
     \end{enumerate}
   \end{prop}
   
   \begin{thm}\label{at2} 
     Let $\cA$ satisfies S-condition.
     \begin{enumerate}
     \item  For every indecomposable object $B\notin\mS(\cA)$ there are atomic objects $\,A_1,A_2,\dots,A_m\,$ and S-objects 
     $\,\lst Sn\,$ such that \ 
     $\bop_{i=1}^mA_i\simeq B\+\big(\bop_{j=1}^nS_j\big)$. These objects are uniquely defined, up to isomorphism and permutation.
     
     \item  $\setsuch{[A]}{ A\in\mA(\cA)\cup\mS(\cA)}$ is a basis of $K_0(\cA)$.
     \end{enumerate}
   \end{thm}
   
   Note that the subgroup generated by $\mA(\cA)$ is isomorphic to the subgroup of $K_0(\hcA)$ with the basis 
   $\mA(\hcA)$ and  the subgroup generated by $\mS(\cA)$ is isomorphic to $K_0(Q\cA)$.     
  
  \begin{proof}
  (1)  Let $\hat{B}\simeq C\+B'$, where $C=\left(\bop_{i=1}^mC_i\right)$ with $C_i\in\mA(\hcA)$ and $B'$ has no direct summands 
  from $\mA(\hcA)$.
  Then $B'\simeq\bop_{U\in\ind\hcA}\mu(U,QB')\sS(U)$ and
  \[\textstyle
   \mu(U,Q\hat{B})=\sum_{i=1}^m\left(\mu(U,QC_i)+\de(U,QC_i)\right) - \left(\sum_{i=1}^m\de(U,QC_i)-\mu(U,QB')\right).
  \] 
  Moreover, $\mu(U,QB')=\de(U,QC)$. Indeed, $\mu(U,QB')\ge\de(U,QC)$, since $\mu(U)\div \mu(U,QC)+\mu(U,QB')$, and
  \begin{align*}
    &\textstyle  Q\hat{B}\simeq\big(QC\+(\bop_U\de(U,QC)U)\big)\+\big(\bop_U(\mu(U,QB')-\de(U,QC))U\big), \\
    \intertext{whence, by Theorem~\ref{loc1}, $B\simeq B_1\+ B_2$, where}
    &\textstyle\hat{B}_1\simeq C\+(\bop_U\de(U,QC)\sS(U)) \\
    \intertext{and}
    &\textstyle \hat{B}_2\simeq \bop_U(\mu(U,QB')-\de(U,QC))\sS(U).
  \end{align*} 
  As $B$ is indecomposable, $B_2=0$, so $\mu(U,qB')=\de(U,QC)$. Let $A_i=\tilde{C}_i$, $A=\bop_{i=1}^mA_i$. 
  Then $Q\hat{A}\simeq QC\+(\bop_U\de(U,QC_i)U)$. 
  Obviously, $\de(U,QC)\le\sum_{i=1}^m\de(U,QC_i)$ and $\mu(U)\div \sum_{i=1}^m\de(U,QC_i)-\mu(U,QB')$. Therefore, there
  is an object $V\in\ob Q\cA$ such that $\hat{V}\simeq\bop_U(\sum_{i=1}^m\de(U,QC_i)-\mu(U,QB'))U$. Let $S=S_V$. Then
  $Q(B\+S)\simeq QA$ and $\hat{B}\+\hat{S}\simeq\hat{A}$. By Theorem~\ref{loc1}, $B\+S\simeq A$. If 
  $V\simeq\bop_{j=1}^nW_j$ with $W_j\in\ind Q\cA$, then $S\simeq\bop_{j=1}^nS_j$, where $S_j=\sS(W_j)$ are S-objects.
  Note that the cores $\co(A_i)$ (hence $A_i$) are uniquely determined: they are direct summands of $\hat{B}$ that are not S-objects. 
  As $Q\cA$ is a \ksc, the objects $W_j$ are uniquely determined. Hence S-objects $S_j$ are also uniquely determined.
  
  \smallskip
  (2) By (1), $\ssuch{[A]}{\mA(\cA)\cup\mS(\cA)}$ is a set of generators of $K_0(\cA)$. So we only have to show that if
  $(\bop_{i=1}^mA_i)\+(\bop_{j=1}^nS_j)\simeq(\bop_{i=1}^kA'_i)\+(\bop_{j=1}^lS'_j)$, where $A_i,A'_i$ are atomic and $S_j,S'_j$
  are S-objects, then $m=k,n=l$ and the objects $A_i$ and $A'_i$, as well as the objects $S_j$ and $S'_j$, are isomorphic up to a
  permutation. Taking cores, we obtain that $\bop_{i=1}^m\co(A_i)\simeq\bop_{j=1}^k\co(A'_j)$, whence $m=k$ and 
  $\co(A_i)\simeq \co(A'_{\si i})$,  hence $A_i\simeq A'_{\si i}$ for some permutation $\si$. As the cancellation holds true in the
  category $\cA$, also $\bop_{j=1}^nS_j\simeq\bop_{j=1}^lS'_j$, whence $\bop_{j=1}^nQS_j\simeq\bop_{j=1}^lQS'_j$. 
  As $QS_j, QS'_j$ are indecomposable and $Q\cA$ is a \ksc, $n=l$ and $QS_j\simeq QS'_{\tau j}$, hence $S_j\simeq S'_{\tau j}$
   for some permutation $\tau$.  
  \end{proof}
  
  \begin{cor}\label{at2c} 
    Suppose that $\cA$ satisfies the S-condition and $\hat{V}$ is indecomposable for every $V\in\ind Q\cA$. Then $K_0(\cA)$
    is a \ksc.
  \end{cor}
  
  Note that this condition is not necessary for $\cA$ to be a \ksc.

  \begin{proof}
   It follows from Theorem~\ref{at2}, since in this case any indecomposable object from $\cA$ is either atomic or 
   an S-object.
  \end{proof}
  
  \begin{exam}\label{nks} 
  \def\hqS{\hat{\qS}}
    We present here an example when $\cA$ is not a \ksc.\!%
    \footnote{\, Perhaps, it is the simplest example. For other examples see \cite[\S\,35]{cr1} and \cite{gud}, where the case of group rings 
    $\mZ_{(p)}G$ is studied.}
    Let $\qK$ be an extension of the field $\qQ$ such that $\dim_\qQ\qK=3$
    and, if $\qS$ is the integral closure of $\qR$ in $\qK$, $\gP\qS$ is a product of $3$ different prime ideals of $\qS$: 
    $\gP\qS=\gP_1\gP_2\gP_3$. For instance, 
    $\qR=\mZ_7=\ssuch{a\in\mQ}{a=m/n, \text{ where } m,n\in\mZ \text{ and } 7\nmid n}$ and $\qK=\mQ(\sqrt[3]{6})$. 
    Then $\qS/\gP_i\simeq\qR/\gP$. Fix isomorphisms $\vi_i:\qS/\gP_i\ito\qR/\gP$ and set 
    $\La=\ssuch{a\in\qS}{\vi_1(a)=\vi_2(a)=\vi_3(a)}$, $N_{ij}=\ssuch{a\in\qS}{\vi_i(a)=\vi_j(a)}$. Let $\cA=\La\lat$. 
    The ring $\hLa$ is local, $\hqS=\qS_1\xx\qS_2\xx\qS_3$, where $\qS_i$ is the $\gP_i$-adic completion of $\qS$, and
    $\hat{N}_{ij}\simeq \hLa_{ij}\+\qS_k$, where $\hLa_{ij}$ is the projection of $\hLa$ onto $\qS_i\xx\qS_j$ and $k\notin\{i,j\}$.
    We set $\mS(\hcA)=\set{\qS_1,\qS_2,\qS_3}$ (actually, this choice is unique).
    The $\hLa$-lattices $\hLa_{ij}$ are indecomposable. Hence the $\La$-lattices $N_{ij}$ are atomic. According to 
    Theorem~\ref{loc1}, there is a $\La$-lattice $M$ such that $\hat{M}\simeq\hLa_{12}\+\hLa_{13}\+\hLa_{23}$ and 
    it is indecomposable. Then $M\+\qS\simeq N_{12}\+N_{13}\+N_{23}$. (It is the decomposition from Theorem~\ref{at2}\,(1).)
    Thus we have decompositions of the same module into direct sums both of $2$ and of $3$ indecomposables and all of them
    are non-isomorphic.
    
    In this example $\hLa$ is a \emph{triad of \dvr{s}} \cite[p.23]{ba2}. So it follows from the calculations there   
    that $\mA(\cA)=\set{N_{12},N_{13},N_{23},\La,\La^*}$, where $\La^*=\ssuch{a\in\qS}{\vi_1(a)=\vi_2(a)+\vi_3(a)}$. 
    Therefore, $K_0(\La\lat)=\mZ^6$.
  \end{exam}
      
   \section{$K_0$, global case}
   \label{kg} 
   
    In this section we suppose that $\qR$ is any Dedekind domain and the \lof\ $\qR$-category $\cA$ satisfies the following 
    condition.
  
     \begin{defin}[Max-condition]\label{max-cond} 
       We say that a \lof\ $\qR$-category $\cA$ \emph{satisfies the Max-condition} if for every indecomposable object 
      $W\in\ob Q\cA$ there is an object $S\in\ob\cA$ such that $W\simeq QS$ and $\De(U)=\End_{\cA^0}S^0$ is a maximal 
      order in the skewfield $Q\cA^0(W^0,W^0)$.
      
       If this condition is satisfied, we fix, for every $W\in\ind Q\cA$, an object $\sS(W)$ such that $\sS(W)\simeq W$
       and $\End_{\cA^0}\sS(W)^0$ is a maximal order,
      and set $\sS(V)=\bop_{i=1}^m\sS(W_i)$ if $V\simeq\bop_{i=1}^mW_i$, where $W_i\in\ind Q\cA$.   
      We set $\mS(\cA)=\ssuch{\sS(W)}{W\in\ind Q\cA}$. We denote by $\sS_\gP(W)$ the image of $\sS(W)$ in $\cA_\gP$
      and $\mS(\cA_\gP)=\ssuch{\sS_\gP(W)}{W\in\ind Q\cA}$. Since $\End_{\cA_\gP^0}S^0_\gP=(\End_{\cA^0} S^0)_\gP$ is a maximal 
      $\qR_\gP$-order if $S=\sS(W)$, $\cA_\gP$ satisfies the Max-condition as well. 
     \end{defin}
     
     \begin{rmk}\label{max-rem} 
      If $\End_{\cA^0}S^0$ is a maximal order and $S=\bop_{i=1}^kS_i$, where $S_i$ are indecomposable, then $QS^0_i$ are simple
      $Q\cA^0$-objects and all $\End_{\cA^0}S^0_i$ are also maximal. It follows from \cite[21.2]{rei}. Therefore,
      the Max-condition is satisfied if for every $A\in\ob\cA^f$ there is an object $S$ such that $QS\simeq QA$ and 
      $\End_{\cA^0}S^0$ is a maximal order.
     \end{rmk}
     
     For instance, Max-condition is satisfied if $\qR$ is excellent and $\cA=\La\mdd$, where $\La$ is a finite $\qR$-algebra. It also
     is satisfied if $\cA=\sw$, the stable homotopy category of polyhedra. In this case all indecomposable objects in $Q\sw$ are $Q S^n$,
     where $S^n$ is an $n$-dimensional sphere, and $\End_\sw S^n=\mZ$. 
        
     \begin{prop}\label{max-s} 
     Suppose that $\cA$ satisfies the Max-condition, $\gP\in\max\qR$.
     \begin{enumerate}
     \item  $\nil\hcA=\hqR_\gP\*_\qR\nil\cA$ and $\,\nil Q\hcA=\hqQ_\gP\*_\qR\nil\cA$. Therefore,
     			$(\hcA_\gP)^0=(\wh{\cA^0})_\gP$ and $(Q\hcA_\gP)^0=Q(\wh{\cA^0})_\gP$.     			   
     \item  The category $\hcA_\gP$ satisfies the Max-condition and the category $\cA_\gP$ satisfies the S-condition.
     \item   For every object $B\in\cA_\gP$ there is an object $A\in\cA$ such that $A_\gP\simeq B$.
     \end{enumerate}
     \end{prop}
     \begin{proof}
      (1) If $\De$ is a maximal order, its center $\qC$ is integrally closed, hence is a Dedekind domain. Let $\qK$ be its field
      of fractions, $\qD=Q\De$ and $\lst\gP k$ be all primes of $\qC$ containing $\gP$. Then $\hDe_\gP\simeq\prod_{i=1}^k\hDe_{\gP_i}$ 
      and $\hat{\qK}_\gP\simeq\prod_{i=1}^k\hat{\qK}_{\gP_i}$ \cite[24.C]{ma}. Since $\qD$ is central over $\qK$, all $\hat{\qD}_{\gP_i}$
      are central simple algebras, so $\hat{\qD}_\gP$ is semisimple. It implies that $\hqQ_\gP\*_\qR\cA^0$ is semisimple, whence
      $\nil Q\hcA=\hqQ_\gP\*_\qR\nil\cA$ and $\nil \hcA=\hqR_\gP\*_\qR\nil\cA$.
     
      (2) If $\De$ is a maximal order in a skewfield $\qD$ and $\hat{\qD}_\gP=\prod_{i=1}^m\Mat(n_i,\qD_i)$, where $\qD_i$ are skewfields,
      $\hDe_\gP$ also splits as $\hDe_\gP\simeq\prod_{i=1}^m\Mat(n_i,\De_i)$, where $\De_i$ is a maximal order in $\qD_i$ \cite[18.8]{rei}. 
      Then $\hat{\qD}_\gP\simeq\bop_{i=1}^m n_i\qD_i^{n_i}$ as $\hat{\qD}_\gP$-module, where all summands are simple modules with the 
      endomorphism rings $\qD_i$. Respectively, $\hat{\De}\simeq\bop_{i=1}^mn_i\De_i^{n_i}$, where the summands have 
      endomorphism rings $\De_i$. Suppose that $\qD=\End_{Q\cA^0}W^0$ and $\De=\End_{\cA^0}S^0$, where $S=\sS(W)$. 
      By Lemma~\ref{11} and Proposition~\ref{12}, $\hat{W}$ and $\hat{S}$ split in the same way, namely, every indecomposable 
      summand of $\hat{W}_\gP$ is of the form $QA$, where $A$ is an indecomposable summand of $\hat{S}_\gP$, and the 
      endomorphism rings of $QA^0$ and $A^0$ are, respectively, $\qD_i$ and $\De_i$ for some $i$. 
	  
	  (3) Obviously, we can suppose $\cA$ semiprime and torsion free.
	  The object $B$ arises from an idempotent $e\in\End_{\cA_\gP}B'_\gP\sbe\End_{Q\cA}QB'$, where $B'\in\ob\cA$. Let 
	  $S=\sS(QB)$,\,$C=S\+B'$,\,$\La=\End_\cA(C)$,\,$M=\cA(C,S)$ and $L=\cA_\gP(C_\gP,B)$. Then $QM\simeq QL$. 
	  Thus we can suppose that $L\sbe M_\gP$ and consider the $\La$-lattice $N=L\cap M$. Note that $N_\gP=L$ and $N_\gQ=M_\gQ$
	  if $\gQ\ne\gP$, so all localizations of $N$ are projective over the localizations of $\La$ and $N$ is projective over $\La$.
	  Therefore $N\simeq\cA(C,A)$ for some $A$ and $A_\gP\simeq B$.
     \end{proof}
     
    Following the proof of item (2) above, we fix, for every $U_\gP\in\ind Q\hcA_\gP$, an object $W\in\ind Q\cA$ such that $U_\gP$ 
    is a direct summand of $\hat{W}_\gP$ and denote by $\sS(U_\gP)$ an indecomposable direct summand of $\wh{\sS(W)}_\gP$ 
    (it is unique up to an isomorphism). We set $\mS(\hcA_\gP)=\ssuch{\sS(U_\gP)}{U_\gP\in\ind Q\hcA_\gP}$ and 
    $\mA(\hcA_\gP)=\ind\hcA_\gP\=\mS(\hcA_\gP)$.
     We call the objects from $\mS(\cA),\mS(\cA_\gP)$ and $\mS(\hcA_\gP)$ the \emph{S-objects} of the corresponding categories.     
     As we have the notion of S-objects in $\hcA_\gP$, we define the set $\mA(\cA_\gP)$ of the \emph{atomic objects} in $\cA_\gP$ as in
      the preceding section.      
      
     Proposition~\ref{max-s}\,(3) and Theorem~\ref{lg}\,(2) imply that, if the Max-condition is satisfied, a genus $G(A)$ of objects from 
     $\cA$ is defined by the object $V=QA$ from $Q\cA$, the (finite) set $\bP$ of prime ideals $\gP$ such that 
     $A_\gP\not\simeq \sS(V)_\gP$ and the localizations $A_\gP$ for $\gP\in\bP$. If $\bP=\varnothing$, $A=\sS(V)$. 
     Moreover, these data can be prescribed  arbitrary, with the only restriction that $QA_\gP\simeq V$ for all $\gP\in\bP$.
      
     \begin{defin}\label{at3} 
       An object $A\in\ob\cA$ is said to be \emph{$\gP$-atomic} if $A_\gP\in\mA(\cA_\gP)$ and $A_\gQ\simeq \sS(QA)_\gQ$ for $\gQ\ne\gP$.
      We denote by $\mA(\gP,\cA)$ the set of $\gP$-atomic objects, set $\mA(\cA)=\bup_{\gP\in\max\qR}\mA(\gP,\cA)$ and call the objects
      from $\mA(\cA)$ \emph{atomic}. 
       
       \smallskip\noindent
       By the remark above, every atomic object from $\mA(\cA_\gP)$ is the $\gP$-localization of a $\gP$-atomic object, so there is a 
       one-to-one correspondence between $\mA(\cA,\gP)$ and $\cA(\cA_\gP)$ (or, the same, $\mA(\hcA_\gP)$). 
       Obviously, atomic objects are indecomposable.
     \end{defin}
     
     \medskip
     We denote by $G\cA$ the set of genera of $\cA$ and define the group $K_0(G\cA)$ as the quotient of the free group with the basis
     $G\cA$ by the subgroup generated by the elements of the form $G(A\+B)-G(A)-G(B)$. We denote by $[G(A)]$ the class of $G(A)$
     in $K_0(G\cA)$. There is a commutative diagram of groups 
     \begin{equation}\label{eg1} 
     \xy<0em,2.2em>
      \xymatrix@R=1ex{ && K_0(\cA_\gP) \ar[r]^\wedge \ar[dd]^Q  & K_0(\hcA_\gP) \ar[dd]^Q \\
      					K_0(\cA) \ar[r]^G & K_0(G\cA) \ar[ur]^{(\ )_\gP} \ar[dr]_Q \\
						&& K_0(Q\cA) \ar[r]^{\wedge_\gP}   & K_0(Q\hcA_\gP)    }
	 \endxy
     \end{equation}
     The arrow $\xarr G$ is surjective by definition. The Max-condition ensures that the arrows $\xarr Q$ are surjective too. 
     
     \begin{thm}\label{at4} 
       If $\cA$ satisfies the Max-condition, the group $K_0(G\cA)$ is a free abelian group with a basis 
       $\dS=\ssuch{[G(A)]}{A\in\mA(\cA)\cup\mS(\cA)}$.
     \end{thm}
     
     Note that the subgroup generated by $\ssuch{[G(A)]}{A\in\mS(\cA)}$ is isomorphic to $K_0(Q\cA)$ and the subgroup generated by
     $\ssuch{[G(A)]}{A\in\mA(\gP,\cA)\cup\mS(\cA)}$ is isomorphic to $K_0(\cA_\gP)$.
     
     \begin{proof}
      Let $\bA$ be the subgroup of $K_0(G\cA)$ generated by the set $\dS$. Suppose first that $A\in\ind\cA$ is such
      that $A_\gQ\simeq \sS(QA)$ for all prime $\gQ$ except a unique $\gP$. By Theorem~\ref{g2}\,(1), $A_\gP$ is also indecomposable. 
      By Theorem~\ref{at2}, either $A_\gP\simeq\sS(QA)_\gP$ and $G(A)=G(\sS(QA))$ or 
      $A_\gP\+S_\gP\simeq \bop_{i=1}^m(A_i)_\gP$, where $S$ is a direct sum of S-objects, and all $A_i$ are $\gP$-atomic. In the last case 
      $G(A\+S)=G(\bop_{i=1}^mA_i)$, so $[G(A)]+[G(S)]=\sum_{i=1}^m[G(A_i)]$ and $[A]\in\bA$.
     
      If $A$ is arbitrary, let $S=\sS(QA)$. As $QS\simeq QA$, there are morphisms $A\xarr\al S\xarr\be A$ such that $\be\al=a1_A$
      and $\al\be=a1_S$ for some non-zero $a\in \qR$. By Theorem~\ref{lg}\,(1), there are objects $A_i (1\le i\le r)$ such that for each 
      $i$ there is at most one prime $\gP_i$ such that $(A_i)_{\gP_i}\not\simeq S_{\gP_i}$ and $A\+(r-1)S\simeq \bop_{i=1}^rA_i$.
      By the preceding consideration, $[G(A_i)]\in\bA$, therefore also $[G(A)]\in\bA$. Thus $\bA=K_0(G\cA)$, i.e 
      $\dS$ generates $K_0(G\cA)$.
      
      Suppose now that $\sum_{i=1}^n[G(A_i)]=\sum_{j=1}^m[G(B_j)]$, where all $[G(A_i)]\in\dS$. Let $\lst Ak$ and $\lst Bl$ be all
      objects from this list that belong to $\mA(\gP,\cA)$. Then in the group $K_0(\cA_\gP)$ all classes $[(A_i)_\gP]$ with $i>k$ and 
      all classes  $[(B_j)_\gP]$ with $j>l$ belong to the subgroup generated by $\mS(\cA_\gP)$. By Theorem~\ref{at2}\,(2), $k=l$ and 
      $(A_i)_\gP=(B_{\si i})_\gP$ for some permutation $\si$, whence $G(A_i)=G(B_{\si i})$. As it is valid for all primes $\gP$, 
      it remains the case when all summands are from $\mS(\cA)$, which is evident.
     \end{proof}
    
     \begin{cor}\label{at5} 
      $K_0(\cA)\simeq \ker G\+K_0(G\cA)$, where $G$ is the homomorphism from the diagram \eqref{eg1}.    
     \end{cor}
     
     Now we have to calculate $\ker G$. We use the relation $\equiv$ and the groups $\cl(A)$ defined in Corollary~\ref{g3}.
     
     \begin{thm}\label{kerG}       
      If $\cA$ satisfies the Max-condition,
      $ \ker G\simeq \bop_{S\in\mS(\cA)} \cl(S) $.
     \end{thm}
     
     Recall that $\cl(S)\simeq \cl(\De)$, where $\De=\End_{\cA^0}S^0$ is a maximal order in a skewfield. If $\De$ is commutative
     (that is, a Dedekind domain), $\cl\De$ is just the group of ideal classes of $\De$. In the arithmetic case, when $\qQ$ is a global
     field, all groups $\cl(S)$ are finite. Thus, if $\ind Q\cA$ is finite (for instance, $\cA=\La\mdd$ for a semiprime $\qR$-order $\La$),
     $\ker G$ is finite.
     
     \begin{proof}
      First we prove a lemma.

     \begin{lem}\label{kg1} 
       Suppose that $\End_{\cA^0}A^0$ is a maximal order, $A_1.A_2\in G(A)$. If $[A_1]=[A_2]$, then $A_1\equiv A_2$ in $G(A)$.
     \end{lem}
     \begin{proof}
      By Proposition~\ref{12}, we can suppose that $\cA$ is semiprime. Let $[A_1]=[A_2]$, i.e. $A_1\+B\simeq A_2\+B$ for some
      object $B$. Let $\La=\End_\cA(A\+B)$, $M=\cA(A\+B,A)$, $M_i=\cA(A\+B,A_i)\ (i=1,2)$ and $N=\cA(A\+B,B)$. 
      Then $\La$ is a semiprime order, $M,M_1,M_2$ and $N$ are right $\La$-lattices and $M_1\+N\simeq M_2\+N$. 
      Note that $\End_\La M_i\simeq\End_\La M\simeq\De$,
      hence $\Ga=\End_\De M\simeq \End_\De M_i$ is a maximal order, which is an overorder of $\La/\ann_\La M$. 
      If $\ann_\La M\ne0$, then $Q\La=\qA_1\xx\qA_2$ such that $\qA_2=Q\ann_\La M$. If $B_1$ is the projection of $N$
      onto $(QN)\qA_1$, then $M_1\+B_1\simeq M_2\+B_1$. Therefore, we can suppose that $M$ is faithful and $\Ga$ is an
      overorder of $\La$. Now $M_1\+N\simeq M_2\+N$ implies $M_1\+N\Ga\simeq M_2\+N\Ga$, where all summands are 
      $\Ga$-lattices. Since $\Ga$ is maximal, all $\Ga$-lattices $L$ with fixed $QL$ are in the same genus \cite[31.2]{cr1}, hence
      all $\Ga$-lattices belong to $\add M$ (it follows from Theorem~\ref{g2}). By Corollary~\ref{g5}\,(3), 
      $M_1\+N\Ga \simeq M_2\+N\Ga$ implies $M_1\equiv M_2$. Returning, by Lemma~\ref{11}, to $\cA$, we obtain that
       $A_1\equiv A_2$.
     \end{proof}
     
     Fix now a set $\dA$ of representatives of indecomposable genera of the category $\cA$. For convenience, we suppose that
     $\dA\spe\dS=\ssuch{\sS(W)}{W\in\ind Q\cA}$. One easily verifies that $\ker G=\ssuch{[A']-[A]}{A'\in G(A)}$. 
     Let $S=\sS(QA)$. By Theorem~\ref{g2}\,(2), there is an object $S'\in G(S)$ such that $A'\+S\simeq A\+S'$, whence 
     $[A']-[A]=[S']-[S]$. Therefore, $\ker G=\ssuch{[S']-[S]}{S\in\mS(\cA),\,S'\in G(S)}$. By Lemma~\ref{kg1}, $[S']-[S]=[S'']-[S]$
     \iff $S'\equiv S''$. Hence, for any fixed $S$, mapping $[S']-[S]$ to $c(S')$, we obtain an isomorphism of the subgroup $\fK(S)$ 
     generated by the differences $[S']-[S]$ with the group $\cl(S)$.  By definition, if $S_1,S_2\in\mS(\cA)$ and $S_1\not\simeq S_2$, 
     then $\cA(S_1,S_2)\in\nil\cA$. It easily implies that $\fK(S_1)\cap\sum_{S_2\not\simeq S_1}\fK(S_2)=0$, which accomplishes the proof.
     \end{proof}
          
     \begin{exam}\label{hered} 
       Let $\cA=\La\lat$, where $\La$ is a hereditary $\qR$-order. As every $\La$-lattice is projective, $K_0(\cA)$ coincides with $K_0(\La)$, 
       the Grothendieck group of projective $\La$-modules. Since $\La$ decomposes just as $Q\La$ and the center of each component of 
       $\La$ is a Dedekind ring \cite[10.8,10.9]{rei}, we may suppose that $Q\La$ is a central simple $\qQ$-algebra. Then $\mS(\cA)$
       consists of a unique lattice $S$. For every prime $\gP$, $Q\hLa_\gP\simeq\Mat(n_\gP,\qF_\gP)$, where $\qF_\gP$ is a skewfield. 
       Let $\Ga_\gP$ be the (unique) maximal $\hqR_\gP$-order in $\qF_\gP$ \cite[12.8]{rei}, $\gM_\gP$ be its maximal ideal. 
       There is an integer $m_\gP\le n_\gP$ such that $\hLa_\gP$ is Morita equivalent to the ring $\qH(m_\gP,\Ga_\gP)$ of 
       $m_\gP\xx m_\gP$ matrices of the form
       \[
        \mtr{ \Ga_\gP & \Ga_\gP & \Ga_\gP & \dots & \Ga_\gP \\
				 \gM_\gP & \Ga_\gP & \Ga_\gP & \dots & \Ga_\gP \\
				 \gM_\gP & \gM_\gP & \Ga_\gP & \dots & \Ga_\gP \\
				 \hdotsfor 5 \\
				 \gM_\gP & \gM_\gP & \gM_\gP & \dots & \Ga_\gP }
       \]
       (see \cite[39.14]{rei}). As $\La_\gP$ is maximal for almost all $\gP$, almost all $m_\gP=1$. This ring, hence also $\hLa_\gP$,
       has $m_\gP$ indecomposable lattices and one of them must be chosen as $\sS(U)$, so $\mA(\hcA_\gP)$ consists of $m_\gP-1\,$ 
       lattices, as well as $\mA(\cA_\gP)$. If $Q\La\simeq \Mat(n,\qD)$, where $\qD$ is a skewfield, then $\End_\La S=\De$ is a maximal 
       order in $\qD$. Let $m=1+\sum_\gP(m_\gP-1)$. Theorems~\ref{at4} and \ref{kerG} imply that $K_0(\La)\simeq \cl(\De)\+\mZ^m$.
     \end{exam}
          
     \begin{exam}\label{sw} 
          Consider the \emph{stable homotopy category of polyhedra} $\sw$. Its objects are \emph{pointed polyhedra}, that is
     finite CW-complexes with a fixed point, and the sets of morphisms are \emph{stable homotopy classes} of continuous maps,
     $\hos(A,B)=\varinjlim \hot(S^nA,S^nB)$, where $\hot(A,B)$ is the set of homotopy classes of continuous maps preserving fixed
     points and $S^nA$ is the $n$-fold suspension of $A$. It is known \cite{coh} that $\sw$ is a fully additive and locally finite 
     $\mZ$-category. The bouquet (one-point union) $A\vee B$ plays the role of direct sum in this category. For a polyhedron $A$ set
     \begin{align*}
      & r_n(A)=\dim_\mQ\mQ\*_\mZ\hos(S^n,A), \text{ where $S^n$ is the $n$-dimensional sphere}, \\
      &\textstyle B(A)=\bigvee_n r_n(A)S^n,\\
      &\textstyle B_0(A)=\bigvee_{r_n(A)>0} S^n.
     \end{align*}
     Note that $\hos(S^n,A)$ is torsion if $n>\dim A$, so $B(A)$ and $B_0(A)$ are finite bouquets of spheres. Moreover,
     $QA\simeq QB(A)$ in the category $Q\sw$ \cite[Prop.\,1.5]{dk}, so we can take $\ssuch{S^n}{n\in\mN}$ for $\mS(\cA)$.
     Then $B(A)=\sS(QA)$. It also implies that the map $QA\mps Q\hat{A}_p$ gives a bijection $\iso Q\sw\ito\iso Q\wh{\sw}_p$.    
     Therefore, every object from $\wh\sw_p$ is of the form $\hat{A}_p$ for an object from $\sw$. In particular, 
     $\mA(\wh{\sw}_p)=\ssuch{\hat{A}_p}{A\in\mA(\sw,p)}$. Thus the $p$-atomic polyhedra are just indecomposable 
     $p$\emph{-primary} polyhedra in the sense of \cite{fr} or \cite{coh}. Recall that a polyhedron $A$ is said to be 
     $p$-\emph{primary} if $A_p\not\simeq B(A)_p$, but $A_q\simeq B(A)_q$ for any prime $q\ne p$. As $g(\mZ)=1$, 
     Theorems~\ref{at4} and \ref{kerG} imply the well-known theorem of Freyd \cite{fr} (see also \cite[Th.\,4.44]{coh}).
     
     \begin{thm}\label{freyd} 
      $K_0(\sw)$ is a free abelian group with a basis consisting of isomorphism classes of spheres and genera of indecomposable
     $p$-primary polyhedra for all prime $p$.    
     \end{thm}
     
      As $QB(A)\in\add QB_0(A)$ and $g(B_0(A))=1$, Theorem~\ref{g2}\,(2) implies the following result proved in \cite[Th.\,2.5]{dk},
      which is a strengthened variant of \cite[Th.\,1.3]{fr}.
     
     \begin{thm}\label{d-k} 
       $G(A)=G(A')$ \iff $A\+B_0(A)\simeq A'\+B_0(A)$.
     \end{thm}
     
     Note also that, using Theorem~\ref{lg}  we obtain the known example of non-uniqueness of decomposition in the category $\sw$. Namely, let
     $A(n)$ denotes the cone of the map $n\nu:S^6\to S^3$, where $\nu$ is the generator of the groups $\pi_6(S^3)\simeq\mZ/24\mZ$.
     One can easily check that there are morphisms $A(1)\xarr\al S^3\vee S^7\xarr\be A(1)$ such that $\al\be=24\ccd1_{S^3\vee S^7}$
     and $\be\al=24\ccd1_{A(1)}$. Then Theorem!\ref{lg}\,(1) implies that $A(1)\vee S^3\vee S^7\simeq A(3)\vee A(8)$ (the polyhedra in the
     right part of this equality are, respectively, $2$-primary and $3$-primary). It is even a homotopic
     equivalence of spaces, since we are in the stable range. All polyhedra in this decompositions are indecomposable. Unlike Example~\ref{nks}, 
     this one is of ``global'' nature. It is essential, since Corollary~\ref{at2c} implies that all localizations $\sw_p$ are Krull--Schmidt categories. 
     \end{exam}

\end{document}